\documentclass[11pt,a4paper,reqno]{amsart}
\usepackage{amsmath, amssymb, bbm, enumerate,  graphicx,tikz, stmaryrd, enumitem,MnSymbol}

\usepackage{thmtools}

\usepackage[pdftex]{hyperref}
\usepackage{cleveref}
\usepackage{bbm}
\usepackage{cases}
\usepackage{array}
\usepackage{tikz-cd}
\usepackage[all]{xy}
\usepackage{framed}

\usepackage[hmarginratio={1:1},vmarginratio={1:1},lmargin=80.0pt,tmargin=90.0pt]{geometry}

\usepackage{tikz}
\tikzset{anchorbase/.style={baseline={([yshift=-0.5ex]current bounding box.center)}}}
\usetikzlibrary{decorations.markings}
\usetikzlibrary{decorations.pathreplacing}
\usetikzlibrary{arrows,shapes,positioning,backgrounds}
\tikzstyle directed=[postaction={decorate,decoration={markings,
    mark=at position #1 with {\arrow{>}}}}]
\tikzstyle rdirected=[postaction={decorate,decoration={markings,
    mark=at position #1 with {\arrow{<}}}}]
    
\usepackage{graphicx}
\usepackage[vcentermath]{youngtab}
\usepackage{nicefrac}

 \newlength{\baseunit}               
 \newcount{\numlines}                
\newtheorem{theorem}[subsubsection]{Theorem}
\newtheorem{lemma}[theorem]{Lemma}
\newtheorem{prop}[theorem]{Proposition}
\newtheorem{corollary}[subsubsection]{Corollary}

\theoremstyle{definition}
\newtheorem{definition}[subsubsection]{Definition}

\newtheorem{remark}[theorem]{Remark}

\newtheorem{example}[theorem]{Example}

\newtheorem{condition}[theorem]{Condition}

\newtheorem*{theorem*}{Theorem}

\newcommand{\cI}{\mathcal{I}}
\newcommand{\tto}{\twoheadrightarrow}
\newcommand{\unit}{{\mathbbm{1}}}

\newcommand{\Ver}{\mathrm{Ver}}
\newcommand{\Ve}{\mathrm{Vec}}

\newcommand{\mN}{\mathbb{N}}
\newcommand{\mZ}{\mathbb{Z}}

\newcommand{\mF}{\mathbb{F}}
\newcommand{\mT}{\mathbb{T}}

\newcommand{\End}{\mathrm{End}}
\newcommand{\Hom}{\mathrm{Hom}}
\newcommand{\Ext}{\mathrm{Ext}}
\newcommand{\Ker}{\mathrm{Ker}}

\newcommand{\id}{\mathrm{id}}

\newcommand{\ev}{\mathrm{ev}}
\newcommand{\co}{\mathrm{co}}

\newcommand{\cA}{\mathcal{A}}
\newcommand{\cB}{\mathcal{B}}
\newcommand{\cC}{\mathcal{C}}
\newcommand{\cD}{\mathcal{D}}

\newcommand{\cV}{\mathcal{V}}

\newcommand{\Fr}{\mathsf{Fr}}

\newcommand{\Fun}{\mathrm{Fun}}
\newcommand{\Rep}{\mathrm{Rep}}

\newcommand{\ex}{\mathit{ex}} 
\newcommand{\faith}{\mathit{faith}} 

\newcommand{\Inna}[1]{\begin{framed} {\tt{\color{red} {\bf Inna:} {#1}}} \end{framed}}

\newcommand{\InnaB}[1]{{#1}}

\begin{document}
\title{Monoidal abelian envelopes and a conjecture of Benson - Etingof}
\author{Kevin Coulembier}
\address{K. C.: School of Mathematics and Statistics, University of Sydney, NSW 2006, Australia}
\email{kevin.coulembier@sydney.edu.au}

\author{Inna Entova-Aizenbud}
\address{I. E.: Department of Mathematics, Ben Gurion University, Beer-Sheva, Israel}
\email{inna.entova@gmail.com}

\author{Thorsten Heidersdorf}
\address{T. H.: Mathematisches Institut, Universit\"at Bonn, Germany}
\email{heidersdorf.thorsten@gmail.com} 



\begin{abstract} We give several criteria to decide whether a given tensor category is the abelian envelope of a fixed symmetric monoidal category. {As a main result we prove that the category of finite-dimensional representations of a semisimple simply connected algebraic group is the abelian envelope of the category of tilting modules.} Benson and Etingof conjectured that a certain limit of finite symmetric tensor categories is tensor equivalent to the finite dimensional representations of $SL_2$ in characteristic $2$. We use our results on the abelian envelopes to prove this conjecture and its variants for any prime $p$.
\end{abstract}

\maketitle

\section{Introduction}

\subsection{Monoidal abelian envelopes}

If a given category $\cD$ is not abelian, a natural question is whether $\cD$ admits an embedding into an abelian category $\cV$ in a minimal or universal way. While there are some general results about the existence of such {\it abelian envelopes} under some assumptions on $\cD$ (e.g. $\cD$ exact), little seems to be known if we require that $\cD$ and $\cV$ are monoidal.

Assume now that $I:\cD\to\cV$ is a fully faithful symmetric monoidal functor from an additive Karoubian
$k$-linear rigid symmetric monoidal category $\cD$ to a tensor category $\cV$ over a field $k$. A pair $(\cV, I: \cD \to \cV)$ is the {\it abelian envelope} of $\cD$ if the functor $I$ induces for any \InnaB{tensor} category $\cA$ an equivalence of the following categories 

\begin{itemize}
\item $\Fun^\ex(\cV,\cA)$, the category of exact symmetric monoidal
$k$-linear 
functors $\cV\to\cA$,
\item $\Fun^\faith(\cD,\cA)$, the category of faithful symmetric monoidal $k$-linear functors $\cD\to\cA$. 
\end{itemize} 

The existence of such monoidal abelian envelopes is non-trivial: in 
\cite{Deligne-interpolation}, an example was given of a category $\cD$ for which 
there exists no monoidal abelian envelope. Explicit realisations of monoidal abelian 
envelopes have so far only been obtained in some examples, see \cite{BE, BEO,
Comes-Ostrik-abelian, EAHS-Deligne}. Existence criteria and formal descriptions of monoidal abelian envelopes for large
classes of categories have been developed in \cite{BEO, CM}.

We give a convenient criterion in Proposition \ref{prop:equiv_abelian} to decide whether a given tensor category is the abelian envelope of a fixed embedded monoidal subcategory.

{We use these criteria to prove in Theorem} \ref{thm:tilt_ab_env} that the category of tilting modules $\cD = Tilt(G)$, where $G$ is a semisimple simply connected algebraic group over a field $k$ of characteristic $p >0$, admits as abelian envelope the category of finite dimensional algebraic representations $\Rep(G)$. {As an application we prove a conjecture of Benson and Etingof in the second part of our paper.}

\subsection{A conjecture of Benson and Etingof} 

A celebrated theorem of Deligne \cite{Del02} asserts that any tensor category of subexponential growth over an algebraically closed field of characteristic $0$ admits a super fibre functor.

It is easy to come up with counterexamples in positive characteristic $p$: taking the quotient of the category of finite dimensional representations of the cyclic group $\mathbb{Z}/p\mathbb{Z}$ by the tensor ideal of negligible morphisms defines the Verlinde category $\Ver_p$. This is a tensor category of subexponential growth which does not admit any super fibre functor for $p> 3$. In \cite{Ostrik}, Ostrik conjectured that any tensor category of subexponential growth admits an analogue of a fibre functor into $\Ver_p$ and proved it for symmetric fusion categories.

{For any $p$} there are counterexamples to Ostrik's conjecture. In fact, Benson and Etingof \cite{BE} \cite{BEO} showed that there exists an infinite ascending chain of finite symmetric tensor categories over $k$ \[ \Ve_k = \Ver_{p^0} \subset \Ver_{p^1} \subset \Ver_{p^2} \subset \cdots \] with fully faithful symmetric tensor embeddings such that $\Ver_{p^n}$ does not admit a tensor functor to a finite tensor category of smaller Frobenius-Perron dimension. Each of these categories $\Ver_{p^n}$ is the abelian envelope of the quotient of $Tilt(SL_2)$ by the tensor ideal of morphisms which factor through an object in the thick ideal generated by the $n$-th Steinberg representation.

A new {question asked in \cite{BE} \cite{BEO} is whether} any finite tensor category admits a fibre functor to $\Ver_{p^{\infty}} = \bigcup_{n \geq 0}\Ver_{p^n}$. 

Another interesting limit of the categories $\Ver_{p^n}$ which has been suggested in \cite{BE} is the following.
Instead of the naive big limit one can consider - mimicking a construction of Entova-Aizenbud, Hinich and Serganova in \cite{EAHS-Deligne} -  a refined limit: one considers a filtration of $\Ver_{p^{n}}$ by abelian subcategories ${\Ver_{p^{n}}^r}$, $r \geq 0$ such that one expects there to be a partial tensor functor $\Ver_{p^{n}}^r \to {\Ver_{p^{n-1}}^r}$ which becomes an equivalence for $n >> r$. If so, one can define $$\mathcal{C}^r =  \varprojlim_{n\to \infty} {\Ver_{p^{n}}^r}.$$ With the obvious embeddings $$\mathcal{C}^r \hookrightarrow \mathcal{C}^{r+1},\qquad\mbox{from $\Ver_{p^{n}}^r\hookrightarrow \Ver_{p^n}^{r+1}$ for $n>>r$}, $$ one can define the colimit $$\mathcal{C}:=\varinjlim_{r \to\infty} \mathcal{C}^r.$$ Then $\mathcal{C}$ is a tensor category with a distinguished object $\bar{L}_1$.

Our main result {in this part} is the following theorem, conjectured for $p=2$ in \cite[Remark 3.14]{BE}.

\begin{theorem} (Lemma \ref{lem:tilt_ff}, Theorem \ref{thm:BE_conj}) There 
exists a fully faithful symmetric monoidal functor $I':Tilt(SL_2) \to 
\mathcal{C}$ sending the standard $2$-dimensional $SL_2$ representation $V$ to 
$\bar{L}_1$. It factors through the {obvious} functor $I:Tilt(SL_2) \to 
\Rep(SL_2)$ and induces an equivalence of symmetric tensor categories 
$$\Rep(SL_2) \stackrel{\sim}{\to} \mathcal{C}.$$
\end{theorem}

\subsection{Structure of the article} Section \ref{sec:prel} contains some preliminary definitions. Section \ref{sec:env} discusses the notion of a monoidal abelian envelope and gives certain criteria to check whether a tensor category $\cV$ is the abelian envelope of a given monoidal category. Section \ref{sec:conjecture} discusses the construction of the limit $\mathcal{C}$ and proves the equivalence $\mathcal{C} \cong \Rep(SL_2)$.

\section{Preliminaries} \label{sec:prel}

\subsection{SM categories} Let $k$ be any field. We adopt the same notion of a $k$-linear symmetric monoidal (SM) category as \cite{Deligne-Milne}. In particular in a such a category $\cC$, the functor $-\otimes -$ is $k$-linear in both variables. Furthermore we have a binatural family of braiding morphisms $\gamma_{XY}:X\otimes Y\stackrel{\sim}{\to} Y\otimes X$ which satisfy the constraints of \cite{Deligne-Milne}. For an object $X\in \cC$, a dual $X^*$ is an object equipped with morphisms $\co_X:\unit\to X\otimes X^*$ and $\ev_X:X^*\otimes X\to\unit$ satisfying the relations in \cite[(0.1.4)]{Del02}. If every object has a dual, we call $\cC$ rigid.

\subsection{SM functors} A $k$-linear functor $F:\cC\to \cC'$ between two SM categories $\cC$ and $\cC'$ is symmetric monoidal if it is equipped with natural isomorphisms $c_{XY}^F:F(X)\otimes F(Y)\stackrel{\sim}{\to}F(X\otimes Y)$ and $\unit\stackrel{\sim}{\to}F(\unit)$ satisfying the usual  compatibility conditions \cite[Definition 1.8]{Deligne-Milne}. In particular we have a commutative diagram
$$\xymatrix{
F(X)\otimes F(Y)\ar[rr]^{\gamma_{F(X)F(Y)}}\ar[d]^{c_{XY}^F}&&F(Y)\otimes F(X)\ar[d]^{c_{YX}^F}\\
F(X\otimes Y)\ar[rr]^{F(\gamma_{XY})}&& F(Y\otimes X).}$$
If $F$ is a monoidal functor and $X^*$ the dual of $X$, then $F(X^\ast)$ is a dual of $F(X)$.

\subsection{Tensor ideals} A tensor ideal $\cI$ in a $k$-linear SM category $\cC$ consists of submodules $\cI(X,Y)\subset \Hom(X,Y)$ for all objects $X,Y$, such that $\cI$ is closed under arbitrary compositions or tensor products with morphisms in $\cC$.
For such a tensor ideal $\cI$, the $k$-linear category $\cC/\cI$ has the same objects as $\cC$ but morphism spaces given by $\Hom(X,Y)/\cI(X,Y)$. Furthermore, $\cC/\cI$ admits a unique SM structure so that the canonical functor $\cC\to\cC/\cI$ is an SM functor.

\subsection{Tensor categories}\label{DefTC} We say that a $k$-linear SM category $\cC$ is a tensor category if

\begin{enumerate} 
\item[(i)] $\cC$ is abelian;
\item [(ii)] the canonical morphism $k \to \End(\unit)$ is an isomorphism;
\item [(iii)] $\cC$ is rigid (every object in $\cC$ is dualisable);
\item [(iv)] every object in $\cC$ has finite length.
\end{enumerate}
Under these assumptions the functor $-\otimes -$ is bi-exact  and the unit object $\unit$ is simple, see for instance Sections 4.2-4.3 in \cite{EGNO}. Note that finite length implies that every morphism space is finite dimensional, \cite[Proposition 1.1]{Del02}. Our tensor categories are therefore locally finite in the sense of \cite[Definition 1.8.1]{EGNO} and the Jordan-H\"older Theorem and Krull-Schmidt Theorem hold. An SM functor between two tensor categories $\cC$ and $\cC'$ is called a tensor functor if it is exact. Such a functor is automatically faithful (see \cite[Proposition 1.19]{Deligne-Milne}).

We note that what we refer to as `tensor categories' are called `symmetric tensor categories' in \cite{BE, EGNO} and as `pre-tannakian categories' in \cite{Comes-Ostrik-abelian, Ostrik}. Also note that in the terminology of \cite{Kevin-tannakian, Del02, Deligne-Milne} tensor categories need not satisfy (iv) above.


\section{Abelian envelope} \label{sec:env}
Fix an arbitrary field $k$.

\subsection{Definition of the abelian envelope}
\subsubsection{}\label{DefIDV}
In what follows $I:\cD\to\cV$ is a $k$-linear SM functor from an additive Karoubian
$k$-linear rigid SM category $\cD$ to a tensor category $\cV$ over $k$.

\begin{definition}\label{def:ab_env}
For a fixed $\cD$, a pair $(\cV, I: \cD \to \cV)$ as above is an {\it abelian envelope} of $\cD$ if for any $k$-linear \InnaB{tensor} category $\cA$ the functor
$$
-\circ I :\;\, \Fun^\ex(\cV,\cA)\;\to\; \Fun^\faith(\cD,\cA)
$$
is an equivalence of categories between
\begin{itemize}
\item $\Fun^\ex(\cV,\cA)$, the category of exact SM 
$k$-linear (tensor)
functors $\cV\to\cA$,
\item $\Fun^\faith(\cD,\cA)$, the category of faithful SM $k$-linear functors $\cD\to\cA$. 
\end{itemize} 
\end{definition}
It follows immediately from the definition that an abelian envelope, when it exists, is unique up to equivalence. Henceforth we thus speak of `the' abelian envelope.

For this section, we focus on the following assumption.
\begin{condition}\label{itms:req_subcat_D}
\mbox{}

\begin{enumerate}
\item\label{itms:req_subcat_D1} $I:\cD\to\cV$ is fully faithful.
\item\label{itms:req_subcat_D2} Any $X\in\cV$ can be presented as the image of a 
map $I(f)$
for some $f:P\to Q$ in $\cD$.
\item\label{itms:req_subcat_D3} For any epimorphism $X\to Y$ in $\cV$ there 
exists a nonzero $T\in\cD$
such that the epimorphism $X\otimes I(T)\to Y\otimes I(T)$ splits.
\end{enumerate}
\end{condition}

\begin{theorem}\label{thm:uni-2}
Under the assumptions of Condition \ref{itms:req_subcat_D}, the functor $I:\cD\to\cV$ yields an abelian envelope of $\cD$. Moreover, under the equivalence in Definition~\ref{def:ab_env}, full functors correspond to full functors.
\end{theorem}

\begin{proof} 
Except for the statement about the full functors, the theorem is proven in \cite[Theorem 9.2.2]{EAHS-Deligne}. Assume now that $F$ is a fully faithful SM $k$-linear functor $F: \cD\to\cA$ and let $F': \cV \to \cA$ be the induced tensor functor. Then $F'$ is a faithful exact symmetric monoidal functor. By rigidity (see also the proof of Proposition \ref{prop:Fr_full}) it is enough to check that for any $M \in \cV$, $$F': \Hom_{\cV}(\unit, M) \longrightarrow \Hom_{\cA}(\unit, F' M)$$ 
is an isomorphism.
 
It follows from condition \eqref{itms:req_subcat_D2} that there exists $T \in I(\cD)$ and an embedding $i:M \hookrightarrow T$. We then have the following commutative diagram with exact rows:
$$\xymatrix{&0 \ar[r]  &\Hom_{\cV}(\unit, M)  \ar[r]^-{i\circ -} \ar[d]^{F'} &\Hom_{\cV}(\unit, T)  \ar[r]  \ar[d]^{F'}  &\Hom_{\cV}(\unit, T/M)  \ar[d]^{F'} \\ &0 \ar[r]  &\Hom_{\cA}(\unit, F' M)  \ar[r]^-{i\circ -} &\Hom_{\cA}(\unit, F' T)  \ar[r] &\Hom_{\cA}(\unit, F' T / F' M).  }$$
All the vertical arrows are monomorphisms. By the assumption that $F=F'\circ I: \cD \to \cA$ is fully faithful, the vertical arrow $$F': \Hom_{\cV}(\unit, T) \longrightarrow \Hom_{\cA}(\unit, F' T)$$ is an isomorphism. Hence by the five-lemma,  $$F': \Hom_{\cV}(\unit, M) \longrightarrow \Hom_{\cA}(\unit, F' M)$$ is an epimorphism, and thus an isomorphism (as required).
 \end{proof}


\begin{lemma}\label{Lem:WIC}If a triple $(I,\cV,\cD)$ satisfies Condition~\ref{itms:req_subcat_D}, it also satisfies the ``weak ideal condition'', that is: 
For any object $X\in\cV$, there exists $T\in I(\cD)$ so that $X \otimes T \in I(\cD)$. 
\end{lemma}
\begin{proof}
Let $X\in\cV$. By \ref{itms:req_subcat_D}\eqref{itms:req_subcat_D2}, there exists $T' \in I(\cD)$ with an epimorphism $q:T'\twoheadrightarrow X$. By \eqref{itms:req_subcat_D3}, there exists $T \in I(\cD)$ such that $$\id_{T} \otimes q: T\otimes T'\twoheadrightarrow T \otimes X$$ splits. Hence $T \otimes X$ is a direct summand of $ T\otimes T' \in I(\cD)$. Since $\cD$ is idempotent complete (Karoubian), we conclude that $T \otimes X \in I(\cD)$, as required.
\end{proof}

\begin{example} Let $k$ be algebraically closed of characteristic 0. Deligne \cite{Deligne-interpolation} constructed families of rigid SM categories $\Rep (S_t)$, $\Rep (GL_t)$, $\Rep (O_t)$ interpolating the finite dimensional representations of the symmetric, the general linear and the orthogonal groups. These categories are not abelian if $t \in \mathbb{N}$ (the $S_t$-case) or $t \in \mathbb{Z}$ (the $GL_t$ and $O_t$-case). Deligne conjectured that in each case there exists a universal tensor category admitting an embedding of the Deligne category, and suggested a construction of such a tensor category. In the $\Rep (S_t)$-case Comes-Ostrik \cite{Comes-Ostrik-abelian} constructed this tensor category as the heart of a suitable $t$-structure on the homotopy category of $\Rep (S_t)$ and proved that it is the abelian envelope  in the sense of Definition \ref{def:ab_env}. The existence of abelian envelopes was also proven in the $GL_t$-case \cite{EAHS-Deligne}, the $O_t$-case \cite{CM} and for the periplectic Deligne category \cite{EnSer}.
\end{example}

\begin{remark}
Since the appearance of the first version of the current paper, the result in \cite[Theorem 9.2.2]{EAHS-Deligne} has been generalised, by making Condition~\ref{itms:req_subcat_D}(3) superfluous, in \cite[Corollary~4.4.4]{PreTop}. It follows that also Theorem~\ref{thm:uni-2} remains valid without Condition~\ref{itms:req_subcat_D}(3).
\end{remark}

\subsection{Alternative criterion}
Consider again $I:\cD\to\cV$ as in \ref{DefIDV}

\begin{prop}\label{prop:equiv_abelian}
Assume that $I:\cD\to\cV$ is fully faithful (Condition \ref{itms:req_subcat_D}\eqref{itms:req_subcat_D1}) and that 
$$\forall T_1, T_2 \in I(\cD),\;\mbox{we have}\; \;\Ext^1_{\cV}(T_1, T_2) =0.$$
Then, Conditions \eqref{itms:req_subcat_D2}, \eqref{itms:req_subcat_D3} in \ref{itms:req_subcat_D} are equivalent to the following condition:

\begin{condition}\label{cond:new_req_simple}
For any simple object $X\in\cV$, there exists $T\in I(\cD)$ so that $X \otimes T \in I(\cD)$. 
\end{condition}

\end{prop}
\begin{proof}
That conditions \eqref{itms:req_subcat_D2}, \eqref{itms:req_subcat_D3} in \ref{itms:req_subcat_D} imply \ref{cond:new_req_simple} is a special case of Lemma~\ref{Lem:WIC}, so we only prove the other direction.

First of all, recall that all objects in $\cV$ have finite length, so, under our assumption of vanishing of $\Ext^1$ in the proposition, Condition \ref{cond:new_req_simple} implies the weak ideal condition in Lemma~\ref{Lem:WIC}.
Let $X\in\cV$. By this weak ideal condition, 
 there exists $T \in I(\cD)$ such that $T \otimes X \in I(\cD)$. Set $T' = T^* 
\otimes T \otimes X$. Then we have maps $$T' \xrightarrow{ev_{T} \otimes \id_X} 
X , \;\;\; X \xrightarrow{coev_{T^*} \otimes \id_X} T'.$$
 These maps are respectively an epimorphism and a monomorphism, hence Condition \eqref{itms:req_subcat_D2} holds.

 
 Next, consider an epimorphism $f:X \to Y$ in $\cV$. Set $K = \Ker(f)$, so that we have a short exact sequence $$ 0 \to K \to X \xrightarrow{f} Y \to 0.$$
By  the weak ideal condition, there exist $T \in I(\cD)$ such that $T \otimes (K\oplus Y) \in I(\cD)$, and hence $T\otimes K,T\otimes Y\in I(\cD)$. This implies that the short exact sequence $$ 0 \to {T} \otimes K \longrightarrow {T} \otimes X \xrightarrow{\id_{{T}} \otimes f} {T} \otimes Y \to 0$$
 splits by assumtion. Hence Condition \eqref{itms:req_subcat_D3} holds, as required.
  This completes the proof of the proposition.
\end{proof}

\begin{remark}
A very minor adaptation of the proof of Proposition \ref{prop:equiv_abelian} also shows that under Condition~\ref{itms:req_subcat_D}\eqref{itms:req_subcat_D1} and the assumption that `for any complex $\Sigma:0\to T_2\to T_1\to T_0\to 0$ in $\cD$ the sequence $I(\Sigma)$ splits whenever it is exact in $\cV$', we have that Conditions \eqref{itms:req_subcat_D2} and \eqref{itms:req_subcat_D3} are equivalent to the weak ideal condition in Lemma~\ref{Lem:WIC}.
\end{remark}

\subsection{Example: algebraic groups in positive characteristic}

Let $G$ denote a semisimple simply connected algebraic group over an 
algebraically closed field $k$ of characteristic $p > 0$. Let $\Rep(G)$ denote 
the tensor category of finite dimensional rational representations of $G$. Its 
ind-completion is a highest weight category (note that the injective objects are 
typically of infinite dimension). The standard $\Delta(\lambda)$ and costandard 
objects $\nabla(\lambda)$ are parametrised by the dominant integral weights 
$\lambda \in X^+$ {(notation as in \cite[Chapter II.2]{Jantzen})}. 
A module $M \in \Rep(G)$ is called tilting if it has a good filtration (a filtration with every quotient a costandard object) and a Weyl filtration (a filtration with every quotient a standard object). We denote by $Tilt(G)$ the full subcategory of tilting modules. It is a rigid monoidal subcategory of $\Rep(G)$, see \cite[Proposition~E.7]{Jantzen}. Its indecomposable objects are (up to isomorphism) parametrised by the dominant integral weights, see \cite[Proposition E.6]{Jantzen}, and we denote by $T(\lambda)$ the corresponding tilting module.

Of particular importance are the Steinberg modules. We define \[ St_r := L((p^r 
- 1)\rho) = T((p^r-1)\rho),\qquad\mbox{for $r\in\mN$}\] {where 
$\rho$ is half the sum of the positive roots for $G$ (notation as in 
\cite[Section II.1.5]{Jantzen}).}

\begin{theorem}\label{thm:tilt_ab_env}
The category $\Rep(G)$ with the natural embedding $Tilt(G) \to \Rep(G)$ satisfies Condition~\ref{itms:req_subcat_D}, in particular it is the abelian envelope of $Tilt(G)$ in the sense of Definition \ref{def:ab_env}.
\end{theorem}

\begin{proof} By \cite[Corollary~E.2]{Jantzen}, the assumptions of Proposition \ref{prop:equiv_abelian} hold, so it is enough to verify Condition \ref{cond:new_req_simple}.

For $L(\lambda)$ with $\langle\lambda, \alpha_0^{\vee}\rangle \leq 2 p^r -1$ the module $St_r \otimes L(\lambda)$ has a good filtration by \cite[Proposition 4.3.1]{BNPS}. Its dual is $(St_r)^{\ast} \otimes L(\lambda)^{\ast} \cong St_r \otimes L(\lambda)^{\ast}$.  The highest weight $-w_0\lambda$ of $L(\lambda)^\ast$ satisfies again the inequality $\langle\lambda, \alpha_0^{\vee}\rangle \leq 2 p^r -1$, so $(St_r \otimes L(\lambda))^{\ast}$ also has a good filtration. Dualising shows that $St_r \otimes L(\lambda)$ has a Weyl filtration as well and is therefore tilting.
Hence for any irreducible module $L$, the module $St_r \otimes L$ is tilting for $r >> 0$. This proves the statement of the theorem.
\end{proof}

\begin{remark} In particular (by Condition \ref{itms:req_subcat_D2}) every 
object $M \in \Rep(G)$ can be embedded in a tilting module. This can be seen 
directly as follows: analogously to the proof of theorem \ref{thm:tilt_ab_env}, 
$M \otimes St_r$ is a tilting module for $r >> 0$. Then $M$ embeds as submodule 
into the tilting module $M \otimes St_r \otimes St_r^*$.
\end{remark}

\begin{remark}
With the same proof, the conclusion of Theorem~\ref{thm:tilt_ab_env} holds for any reductive group $G$ for which (i) $(p-1)\rho$ is a weight and (ii) every $\lambda\in X^+$ has the property that $St_r\otimes L(\lambda)$ is a tilting module for all but finitely many $r\in\mN$.
\end{remark}

\section{On a conjecture of Benson and Etingof} \label{sec:conjecture}
Let $k$ denote an algebraically closed field of characteristic $p>0$.

\subsection{The abelian envelopes $\Ver_{p^n}$}\label{SecC2n}

As before let $Tilt(SL_2)$ denote the symmetric rigid monoidal category of 
$SL_2$-tilting modules over $k$. The indecomposable tilting modules are the 
$T_m$, $m \geq 0$, the tilting modules with highest weight $m\omega$ 
{(where $\omega=\rho$ is the fundamental weight for $SL_2$)}. Clearly 
$T_0 = k$ and $T_1 = V$, the $2$-dimensional standard representation of $SL_2$. 
We briefly review the tensor ideals in $Tilt (SL_2)$, for a full treatment we 
refer to \cite{Selecta}. For $n>0$, denote by $\mathcal{I}_n$ the tensor ideal 
generated by the identity morphism of $T_{p^n-1}$.
This gives a strictly descending infinite chain  \[ Tilt(SL_2)=:\cI_0\supset \mathcal{I}_1 \supset \mathcal{I}_2 \supset \cdots,\qquad\mbox{with}\;\cap_i\mathcal{I}_i=0. \] Let $\mathcal{T}_n = Tilt(SL_2)/\mathcal{I}_n$.

\begin{theorem} \cite[Theorem 1.2 and 1.3]{BEO} There exists an ascending chain of tensor categories over $k$ \[ \Ver_p \subset \Ver_{p^2} \subset \Ver_{p^3} \subset \Ver_{p^4} \subset \cdots \] with fully faithful tensor embeddings such that $\Ver_{p^n}$ is the  abelian envelope of $\mathcal{T}_{n}$ in the sense of Definition \ref{def:ab_env}.
\end{theorem}

Note that $\Ver_p$ is equal to $Tilt(SL_2)/\mathcal{I}_1$, i.e. it is the semisimplification of $Tilt(SL_2)$, introduced as the universal Verlinde category in \cite{Ostrik}.
We will consider objects in $\Ver_{p^m}$ as objects in $\Ver_{p^n}$ for $m<n$ using the embedding in the theorem.

By their property of abelian envelopes, we have monoidal functors $Tilt(SL_2)\to \Ver_{p^n}$ for each $n>0$, and we write the image of the indecomposable tilting modules as $T_i\mapsto \mT_i^{[n]}$. In particular, $\mT^{[n]}_i=0$ for $i\ge p^n-1$ and $\mT^{[n]}_0=\unit$.

We denote by $\Lambda^{[n]}$ the set of integers $a$ with $0\le a< p^{n}-p^{n-1}$. For any integer $a$, we write its $p$-adic expansion as
$$a=\sum_{i\ge 0}a_ip^i.$$
It is proved in \cite[Theorem~4.42]{BEO} that the isomorphism classes of simple objects in $\Ver_{p^n}$ can be labelled by $\Lambda^{[n]}$ as follows:
\begin{equation}\label{LabelSimp}L_a^{[n]}\;:=\; \mT^{[1]}_{a_{n-1}}\otimes \mT^{[2]}_{a_{n-2}}\otimes\cdots\otimes \mT^{[n]}_{a_0},\quad\mbox{for $a\in\Lambda^{[n]}$}.\end{equation}
For $a\in\Lambda^{[n-1]}\subset\Lambda^{[n]}$, we can see that the simple object $L_a^{[n-1]}$ of $\Ver_{p^{n-1}}$ remains simple $\Ver_{p^n}$. We stress that $L_a^{[n-1]}\not= L_a^{[n]}$, but rather $L_a^{[n-1]}= L_{pa}^{[n]}$. We will use freely, see \cite[Proposition~4.59(ii)]{BEO}, that $\Ver_{p^{n-1}}$ is the Serre subcategory of $\Ver_{p^n}$ generated by the simple objects $L_a^{[n-1]}$ (or $L_{pa}^{[n]}$) with $a\in\Lambda^{[n-1]}$.

The category $\Ver_{p^n}$ is filtered by the full topologising subcategories $\Ver_{p^n}^r$, $r \geq 0$, which consists of subquotients of direct sums of $(L_1^{[n]})^{\otimes j}$, $j \leq r$. That
$$\Ver_{p^n}\;=\;\bigcup_{r>0} \Ver_{p^n}^r$$
follows from the fact that every object in $\Ver_{p^n}$ is a quotient of a direct sum of objects $(L_1^{[n]})^{\otimes j}$ (since Condition~\ref{itms:req_subcat_D} is satisfied).
The notation $\Ver_{p^n}^r$ also appears briefly in \cite[Remark~4.46]{BEO}. However, by definition (and Corollary~\ref{Lem2}(1) below) the category from \cite[Remark~4.46]{BEO} is the Serre subcategory generated by $\Ver_{p^n}^r$ as defined above.


\subsection{The Grothendieck rings}

We start by some observations concerning the Grothendieck rings of $\Ver_{p^n}$ that will be needed later on. 

\begin{lemma}
\label{Lem1}
Fix $n\in\mZ_{>0}$ and $a,b\in \Lambda^{[n]}$.
\begin{enumerate}
\item If $a<p^{n}-p^{n-1}-1$, then $[L_1^{[n]}\otimes L_a^{[n]}:L_{a+1}^{[n]}]=1$ and $[L_1^{[n]}\otimes L_a^{[n]}:L_{b}^{[n]}]=0$ whenever $b>a+1$.
\item If $n>1$ and $a<p^{n-1}-p^{n-2}-1$ then
$$[L_1^{[n]}\otimes L_a^{[n]}:L_b^{[n]}]=\begin{cases} 0&\mbox{if } b\not\in \Lambda^{[n-1]},\\
[L_1^{[n-1]}\otimes L_a^{[n-1]}:L_b^{[n-1]}]&\mbox{if } b\in \Lambda^{[n-1]}.
\end{cases}$$
\end{enumerate}
\end{lemma}
\begin{proof}

We rewrite a special case of \cite[Corollary~4.50]{BEO} as follows. 
If $a_0<p-1$, we have 
\begin{equation}\label{Eq1}[L^{[n]}_1\otimes L_a^{[n]}:L_b^{[n]}]\;=\;\begin{cases}
1&\mbox{if $b=a+1$}\\
1&\mbox{if $b=a-1$ and $a_0>0$}\\
0&\mbox{otherwise.}
\end{cases}\end{equation}
On the other hand, if $n>1$ and $a_0=p-1$, we have
\begin{equation}\label{Eq2}[L^{[n]}_1\otimes L_a^{[n]}:L_b^{[n]}]-2\delta_{b,a-1}\;=\;\begin{cases}[L_1^{[n-1]}\otimes L_{\frac{a+1-p}{p}}^{[n-1]} :L_{\frac{b}{p}}^{[n-1]} ]&\mbox{if $p$ divides $b$}\\
0&\mbox{if $p$ does not divide $b$.}
\end{cases}\end{equation}

The first formula proves part (1) instantly for the case $a_0<p-1$. For the remaining case, we perform induction on $n$. If $n=1$, then $a=a_0$ is by assumption smaller than $p-1$. Assume that the claim is proved for $n-1$, then it follows quickly for $n$ from the second displayed formula. This concludes the proof of part (1).

Now we prove part (2). That the multiplicity is zero whenever $b\not\in \Lambda^{[n-1]}$ follows from part (1), since the latter condition implies in particular that $b>a+1$. If $a_0<p-1$, the fact that the right-hand side in \eqref{Eq1} does not depend on $n$ proves the remaining multiplicities in part (2). Finally, the case $a_0=p-1$ can again be proved by induction on $n$ by \eqref{Eq2}.
\end{proof}

\begin{corollary}\label{Cor:Simp}
\label{Lem2}
\begin{enumerate}
\item For an arbitrary $r\in\mZ_{>0}$, the simple objects in $\Ver_{p^n}^r$ are precisely the simple objects $L_{m}^{[n]}$ with $m\le r$ (and~$m\in\Lambda^{[n]}$).
\item Consider $i< p^{n-1}-p^{n-2}$ and $b\in \Lambda^{[n]}$. We have
$$[\otimes^iL^{[n]}_1:L_b^{[n]}]=\begin{cases}
0&\mbox{if }  b\not\in\Lambda^{[n-1]},\\
[\otimes^iL^{[n-1]}_1:L_b^{[n-1]}]&\mbox{if }b\in\Lambda^{[n-1]}.
\end{cases}$$
\end{enumerate}
\end{corollary}
\begin{proof}
Part (1) follows by iteration from Lemma~\ref{Lem1}(1). 

We prove part (2) by induction on $i$. If $i=1$ the claim is a tautology. Assume that $i>1$ and the claim is true for $i-1$. Since $-\otimes-$ is exact, and using part (1), we have
$$[\otimes^iL^{[n]}_1:L_b^{[n]}]\;=\;\sum_{a< i}[\otimes^{i-1}L^{[n]}_1:L_a^{[n]}][L_1^{[n]}\otimes L_a^{[n]}:L_b^{[n]}].$$
By assumption, we have $a< p^{n-1}-p^{n-2}-1$ when $a<i$, so $a\in \Lambda^{[n-1]}$. It then follows from the induction hypothesis on $i-1$ and Lemma~\ref{Lem1}(2) that
$$[\otimes^iL^{[n]}_1:L_b^{[n]}]\;=\;\sum_{a< i}[\otimes^{i-1}L^{[n-1]}_1:L_a^{[n-1]}][L_1^{[n-1]}\otimes L_a^{[n-1]}:L_b^{[n-1]}]$$
if $b\in\Lambda^{[n-1]}$ and zero otherwise.

Applying part (1) again, we find that the right-hand side equals $[\otimes^iL^{[n-1]}_1:L_b^{[n-1]}]$, which concludes the proof.
\end{proof}

We conclude with an observation about first extensions we will need.
\begin{lemma}\label{LemExt10}
For each $a\in\Lambda^{[n]}$, we have
$$
\dim_k\Ext^1_{\Ver_{p^n}}(L_a^{[n]},\unit)=\begin{cases}
1&\mbox{ if $a=p^i+p^{i-1}(p-2)$ for some $1\le i\le n-1$;}\\
0&\mbox{ otherwise.}
\end{cases}
$$
\end{lemma}
\begin{proof}
For $p>2$, this can be obtained from \cite[Proposition~4.65]{BEO}. If $p=2$, it follows from \cite[Proposition~4.12]{BE}\footnote{In the latter source, the notation $X_S$ ($S \subset \{1,\ldots, n-1\}$) stands for the simple module $L_a^{[n]}$ where  $p=2$ and $a =\sum_{j \in S} 2^{n-1-j}$. In particular, $X_{n-1}$ stands for $L_1^{[n]}$ in our notation.}.
\end{proof}

\subsection{The Frobenius functor}\label{ssec:Fr}

We discuss the Frobenius functor from \cite{Kevin-tannakian, EO, Ostrik} acting on $\Ver_{p^n}$. As defined in \cite{EO}, the Frobenius functor on a tensor category $\cC$ is an $\mF_p$-linear symmetric monoidal functor 
$$\Fr:\cC\to\cC\boxtimes\Ver_p,$$
where the right-hand side is a special case of Deligne's tensor product (since $\Ver_p$ is semisimple), explained in \cite[\S2.2]{Ostrik}.
As explained in \cite{Kevin-tannakian, EO}, in case $\cC=\Rep G$ for an affine group scheme $G$ over $k$, this reduces to the ordinary Frobenius twist in the following sense. Firstly, $\Fr$ takes values in $\Rep G=\Rep G\boxtimes \Ve\subset \Rep G\boxtimes\Ver_p$ and $\Fr$ is then the functor which sends a representation $V$ to the subrepresentation $V^{(1)}$ in the symmetric power $S^pV$ spanned by the classes of all vectors $v\otimes  v\otimes\cdots \otimes v\in V^{\otimes p}$, for $v\in V$.

 We thus consider
\begin{equation}\label{OrigFr}\Fr:\Ver_{p^n}\;\to\; \Ver_{p^n}\boxtimes\Ver_p.\end{equation}
                                                                                
As observed in \cite[Remark~4.46]{BEO}, when we restrict $\Fr$ to $\Ver_{p^n}^r$ for $r< p^{n-1}$, then it actually factors as
$$\Fr:\Ver_{p^n}^r\;\to\; \Ver_{p^n}\cong \Ver_{p^n}\boxtimes \Ve \subset \Ver_{p^n}\boxtimes\Ver_p,$$
and more concretely, by \cite[Corollary~4.45]{BEO} we have for $b\in\Lambda^{[n-1]}$
\begin{equation}\label{FrSimple}\Fr(L_b^{[n]})\;\cong\;L_b^{[n-1]}=L_{pb}^{[n]}.\end{equation}
Consequently, the restriction of $\Fr$ to $\Ver_{p^n}^r$ lands in the Serre subcategory $\Ver_{p^{n-1}}\subset \Ver_{p^n}$. 

We will henceforth interpret our restriction of $\Fr$ as a functor
\begin{equation}\label{FrWeNeed}\Fr:\Ver_{p^n}^r\;\to\; \Ver_{p^{n-1}}.\end{equation}

Although one cannot say $\Fr$ is monoidal anymore, seeing that $\Ver_{p^n}^r \subset \Ver_{p^n}$ is not a monoidal subcategory, we can still state that $\Fr$ is ``locally monoidal'': that is, $$\Fr(X) \otimes \Fr(Y) \cong \Fr(X \otimes Y)$$ for any $X, Y \in \Ver_{p^n}^r$ such that $X\otimes Y \in \Ver_{p^n}^r$, too. This isomorphism descends from the monoidal functor $\Fr$ as in \eqref{OrigFr}, so it is natural in both $X$ and $Y$ (as long as all the relevant tensor products belong to $\Ver_{p^n}^r$). See also \cite[Remark 4.46]{BEO}.

It will be useful to have the following technical lemma. Note that it is a slight generalisation of the observation in \cite[Proposition~5.1]{EO} that $\Fr$ commutes with tensor functors.
\begin{lemma}\label{StupidLemma}
Consider tensor categories $\cC,\cV$ and a topologising subcategory $\cC^0\subset\cC$. Consider an object $X\in\cC$ with $X^{\otimes p}\in\cC^0$ and an exact functor $F:\cC^0\to \cV$ with the property that $F(X^{\otimes p})=Y^{\otimes p}$ for some $Y\in\cV$ such that the diagram
$$\xymatrix{\End_{\cC}(X^{\otimes p})\ar[rr]^F&&\End_{\cV}(Y^{\otimes p})\\
&kS_p\ar[ul]\ar[ur]}$$
is commutative. Then $\Fr(X)\in \cC^0\boxtimes \Ver_p\subset \cC\boxtimes \Ver_p$ and $(F\boxtimes\mathrm{Id})(\Fr(X))\cong \Fr(Y)$.
\end{lemma}
\begin{proof}
By definition in \cite[\S 3-5]{EO}, every component of $\Fr(X)$ (where the components live in $\cC$ and are labelled by the simple objects in $\Ver_p$) is obtained as a subquotient of $X^{\otimes p}$ via kernels and images of morphisms in the image of $kS_p\to\End(X^{\otimes p})$.
\end{proof}


\begin{remark}
Recall that $\Fr$ is not $k$-linear. In fact, it follows easily that $\Fr(\lambda f)=\lambda^p\Fr(f)$, for $\lambda\in k$ and $f$ some morphism. It is therefore possible to interpret $\Fr$ as $k$-linear by twisting the $k$-linear structure on the target category via the Frobenius morphism $\lambda\mapsto \lambda^p$ of $k$ (while keeping the same abelian monoidal category), see also \cite[\S 3.3]{Ostrik}. Since we assume that $k$ is algebraically closed and hence perfect, the Frobenius morphism is an isomorphism and in particular this twist does not change the dimensions of the morphism spaces.

Alternatively, again because $k$ is perfect, we can twist the $k$-linear structure on the source category via the inverse of the Frobenius morphism: $\lambda\mapsto \lambda^{1/p}$. That is our preferred method in this paper.
Concretely, we will henceforth `normalise' or `twist' our categories {\em such that $\Ver_{p^{n-1}}\hookrightarrow \Ver_{p^n}$ is not $k$-linear, but \eqref{FrWeNeed} is $k$-linear}. 

Note that any such `twist' of $\Ver_{p^n}$, via $k\to k:\lambda\mapsto \lambda^{p^i}$ for $i\in\mZ$, is (non-canonically) equivalent to the original, for instance because the same is true for $Tilt(SL_2)$ and we can define $\Ver_{p^n}$ as `the' algebraic envelope of a quotient of the former.
\end{remark}

\begin{remark}
We can also define directly the internal Frobenius twist $\Fr_{\mathrm{in}}:\Ver_{p^n}\to\Ver_{p^n}$ from \cite[Definition~4.2.2]{Kevin-tannakian}, which can easily be seen to restrict to the same functor $\Ver_{p^n}^r\to\Ver_{p^n}$ for $r< p^{n-1}$. However, this more direct construction of \eqref{FrWeNeed} has the drawback that $\Fr_{\mathrm{in}}$ is in general not monoidal.
\end{remark}

\subsection{Stabilisation of the categories \texorpdfstring{${\Ver_{p^{n}}^r}$}{Ver}}
We study further the functor in equation~\eqref{FrWeNeed}.
\begin{theorem}\label{thm:Fr_loc_equiv}
The Frobenius twist $\Fr$ restricts to a $k$-linear equivalence $\Fr: {\Ver_{p^{n}}^r} \to {\Ver_{p^{n-1}}^r}$ for $n>4r>0$.
\end{theorem}
\begin{proof}
 The proof will be done in several steps and for the sharper bound $2r<p^{n-1}-p^{n-2}$. The bound in the theorem follows since $n$ is always bounded by $2(p^n-p^{n-1})$. We will prove in Corollary \ref{cor:Fr_exact_faithful} that $\Fr: {\Ver_{p^{n}}^r} \to {\Ver_{p^{n-1}}}$ is faithful {and its image lands in ${\Ver_{p^{n-1}}^r}$}; in Proposition \ref{prop:Fr_full} we prove that it is full, and in Corollary \ref{cor:Fr_ess_surj} {we prove that $\Fr: {\Ver_{p^{n}}^r} \to {\Ver_{p^{n-1}}^r}$ is essentially surjective}.
\end{proof}



We begin by showing that $\Fr:{\Ver_{p^{n}}^r} \to {\Ver_{p^{n-1}}}$ is faithful and exact. 
\begin{definition}\label{Def:Emiddle}
 We say that a functor $F$ between abelian categories is ``exact in the middle'' if it maps every short exact sequence $X\hookrightarrow Y\tto Z$ to an exact sequence $F(X)\to F(Y)\to F(Z)$.
\end{definition}

The following result has been proved in \cite[Proposition~3.6]{EO}:
\begin{lemma}\label{lem:Fr_Emiddle}
 The functor $\Fr$ is exact in the middle.
\end{lemma}

We denote the length of an object $X$ in an abelian category by $\ell(X)\in\mN\cup\{\infty\}$.

\begin{lemma}\label{LemTriv}
Consider two abelian categories $\cA,\cB$, an additive functor $F:\cA\to\cB$ which is exact in the middle, and a set $\{C_i\,|\,i\in I\}$ of objects in $\cA$ of finite length. If $F$ sends simple constituents of the $C_i$ to simple objects and $\ell(F(C_i))=\ell(C_i)$ for every $i\in I$, then $F|_{\cC}:\cC\to\cB$ is exact and faithful, with $\cC$ the topologising subcategory of $\cA$ generated by $\{C_i\}$. 
\end{lemma}
\begin{proof}
The assumption on the action of $F$ on simple objects and the exactness in the middle imply that $\ell(F(M))\le \ell(M)$, for each $M\in\cC$. Now let $C$ be a finite direct sum of objects $C_i$, and consider a short exact sequence
$$0\to X\to C\to Y\to 0$$
in $\cA$ (and hence in $\cC$). By using $\ell(FC)=\ell(C)=\ell(X)+\ell(Y)$ and exactness in the middle, we find
$$\ell(FX)+\ell(FY)\ge \ell(X)+\ell(Y).$$
Since we have already established that $F$ can only decrease the length of objects in $\cC$ we find that $F$ preserves the length of each subobject or quotient of an object $C$ as above. It now also follows easily that the action of $F$ on the short exact sequence must be exact.

 We can now proceed in the exact same way (first by replacing $C$ by an arbitrary subobject and in a second step by an arbitrary subquotient) to prove that $F$ acts exactly on each short exact sequence in $\cC$. That $F$ is faithful on $\cC$ then follows from the fact it is exact and preserves length. 
\end{proof}

\begin{corollary}\label{cor:Fr_exact_faithful}
If $0<r<p^{n-1}-p^{n-2}$, the Frobenius twist $\Fr: {\Ver_{p^{n}}^r}\to{\Ver_{p^{n-1}}}$ is
an exact faithful functor, and its image sits in ${\Ver_{p^{n-1}}^r}$.
\end{corollary}

\begin{proof}
By equation~\eqref{FrSimple} and the fact that {$\Fr$ is ``locally monoidal'' as explained in Section \ref{ssec:Fr}, we have, for every $0 \leq i \leq r$, that $\Fr(\otimes^i L_1^{[n]})\cong \otimes^i L_1^{[n-1]}$. Furthermore, by Lemma \ref{lem:Fr_Emiddle}, $\Fr$ is exact in the middle.} By Lemma~\ref{LemTriv}, in order to show that $\Fr$ is exact and faithful, it thus suffices to show that
$$\ell(\otimes^i L_1^{[n]})\,=\,\ell(\otimes^i L_1^{[n-1]}),\qquad\mbox{for }0\le  i \leq r.$$
This is a direct consequence of Corollary~\ref{Lem2}(2).

{Having shown that $\Fr$ is exact, and using the fact that $\Fr(\otimes^i L_1^{[n]})\cong \otimes^i L_1^{[n-1]}$, we immediately conclude that the image of $\Fr$ sits in ${\Ver_{p^{n-1}}^r}$.}
\end{proof}

We will need the following auxiliary result.
\begin{lemma}\label{Lem:Ext} Consider $0<s<{p^{n-1} - p^{n-2}}$ and let $L$ be a simple object in ${\Ver_{p^{n}}^{s}}$. Then $\Fr$ induces an isomorphism
$$\Ext^1_{{\Ver_{p^{n}}^{s}}}(L,\unit)\stackrel{\sim}{\to}\Ext^1_{{\Ver_{p^{n-1}}^s}}(\Fr L,\unit).$$
\end{lemma}
\begin{proof}
{By Corollary \ref{Cor:Simp}(1), we have: 
$L = L_a^{[n]}$ for some $a \leq s <p^{n-1} - p^{n-2}$. This implies that $a \in \Lambda^{[n-1]}$, so $\Fr(L_a^{[n]}) = L_a^{[n-1]}$ by Equation~\eqref{FrSimple}. 

By Lemma~\ref{LemExt10}, we have
$$\Ext^1_{{\Ver_{p^n}}}(L_a^{[n]},\unit)\cong \Ext^1_{{\Ver_{p^{n-1}}}}(\Fr L_a^{[n-1]},\unit).$$

Furthermore, Lemma~\ref{LemExt10} and equation~\eqref{LabelSimp} show that the above $\Ext^1$ spaces are non-zero precisely when $L_a^{[n]} = \mathbb{T}^{[j]}_1 \otimes \mathbb{T}^{[j+1]}_{p-2}$ for some $2 \leq j \leq n-1$. In that case, the $\Ext^1$ spaces are $1$-dimensional.
Of course, since $a\leq s$ we have: $p^{n-j-1}(2p-2) \leq s$. 

It therefore suffices to prove that any extension in  ${\Ver_{p^{n}}}$ of $\unit$ and $L_{a}^{[n]} = \mathbb{T}^{[j]}_1 \otimes \mathbb{T}^{[j+1]}_{p-2}$,
\begin{itemize}
\item [(a)] exists in ${\Ver_{p^{n}}^{s}}$;
\item [(b)] is not split by $\Fr$.
\end{itemize}

We let $E$ be the unique (up to isomorphism) non-split extension of $\unit$ and $\mathbb{T}^{[j]}_1 \otimes \mathbb{T}^{[j+1]}_{p-2}$ in $\Ver_{p^{n}}$.  
In \cite[Example 4.47]{BEO}, it is shown that the object $\mathbb{T}^{[j+1]}_{2p-2}$ in $\Ver_{p^{n}}$ has simple constituents $\unit, \mathbb{T}^{[j]}_1 \otimes \mathbb{T}^{[j+1]}_{p-2}, \unit$, with $\unit$ appearing both as a subobject and a quotient. 
Since
$$\Hom_{SL_2}(k, T_{2p-2})\cong k\cong \Hom_{SL_2}(T_{2p-2},k),$$
see for instance \cite[Lemma~5.3.3]{Selecta}, it follows that $\unit$ appears at most once in the top and socle of  $\mathbb{T}^{[j+1]}_{2p-2}$. Furthermore, $\Ext^1_{{\Ver_{p^{n}}}}(\unit, \unit)=0$, so we conclude that 
$soc(\mathbb{T}^{[j+1]}_{2p-2})=\unit$ and $\mathbb{T}^{[j]}_1 \otimes \mathbb{T}^{[j+1]}_{p-2}$ constitutes the next layer of the socle filtration.

Hence $E$ is a subobject of $\mathbb{T}^{[j+1]}_{2p-2}$ in $\Ver_{p^{n}}$. The object $\mathbb{T}^{[j+1]}_{2p-2}$ is a direct summand in $\left(L^{[n]}_{(p-1)p^{n-j-1}}\right)^{\otimes 2} = \left(\mathbb{T}^{[j+1]}_{p-1}\right)^{\otimes 2}$. The requirement $a\leq s$ implies that $2p^{n-j-1}(p-1) \leq s$, so $L^{[n]}_{(p-1)p^{n-j-1}} \in \Ver^{\lfloor s/2 \rfloor}_{p^n}$ and hence $ \left(L^{[n]}_{(p-1)p^{n-j-1}}\right)^{\otimes 2} \in\Ver_{p^{n}}^{s}$. Thus both $\mathbb{T}^{[j+1]}_{2p-2}$ and $E$ are in $\Ver_{p^{n}}^{s}$.

The functor $$\Fr: {\Ver_{p^{n}}^{s}} \to  {\Ver_{p^{n-1}}^{s}}$$ is exact, and locally monoidal so $\Fr\left(\left(\mathbb{T}^{[j+1]}_{p-1}\right)^{\otimes 2}\right) = \left(\mathbb{T}^{[j]}_{p-1}\right)^{\otimes 2}$, by equation~\eqref{FrSimple}. Thus $\Fr(\mathbb{T}^{[j+1]}_{2p-2})$ is a direct summand of $\left(\mathbb{T}^{[j]}_{p-1}\right)^{\otimes 2} \in {\Ver_{p^{n-1}}^{s}}$ with same composition factors as $\mathbb{T}^{[j]}_{2p-2} \in {\Ver_{p^{n-1}}^{s}}$, by equation~\eqref{FrSimple}. By \cite[Proposition 4.49]{BEO} the only summand of $\left(\mathbb{T}^{[j+1]}_{p-1}\right)^{\otimes 2} \in {\Ver_{p^{n-1}}^{s}}$ with $\unit$ in the socle is $\mathbb{T}^{[j]}_{2p-2}$, hence $\Fr(\mathbb{T}^{[j]}_{2p-2})=\mathbb{T}^{[j]}_{2p-2}$.} 

This implies in particular that $\Fr(E)$ does not split.\end{proof}

We now prove that $\Fr$ is full.

\begin{prop}\label{prop:Fr_full} 
 If $0<2r<{p^{n-1}-p^{n-2}}$, the Frobenius twist $\Fr: {\Ver_{p^{n}}^r}\to{\Ver_{p^{n-1}}}$ is full.
\end{prop}
\begin{proof}
We use a variation of the proof of \cite[Proposition 5.2.1]{EnSer}.

We have already established that $\Fr: {\Ver_{p^{n}}^r}\to{\Ver^r_{p^{n-1}}}$ is faithful. So we only need to check that 
$$ \dim\Hom_{ {\Ver_{p^{n}}^r}}(N, N')  = \dim \Hom_{ {\Ver_{p^{n-1}}^r}}(\Fr(N), \Fr(N'))$$ for any $N , N' \in  {\Ver_{p^{n}}^r}$. Now, 
\begin{align*} \dim\Hom_{ {\Ver_{p^{n}}^r}}(N, N') & =  \dim\Hom_{ {\Ver_{p^{n}}^{2r}}}(\unit, N^*\otimes N'),\\
  \dim \Hom_{ {\Ver_{p^{n-1}}^r}}(\Fr(N), \Fr(N')) &=  \dim\Hom_{ {\Ver_{p^{n-1}}^{2r}}}(\unit, \Fr(N^*\otimes N')).
\end{align*}



It is enough to prove that the following statement holds: 
 $$ \forall M \in  {\Ver_{p^{n}}^{2r}}, \;\; \dim\Hom_{ {\Ver_{p^{n}}^{2r}}}(\unit, M)= \dim\Hom_{ {\Ver_{p^{n-1}}^{2r}}}(\unit, \Fr(M)).$$
 
 We will prove this by induction on the length of $M$. If $\ell(M)=1$, the statement follows from equation~\eqref{FrSimple}. Now we assume the equation holds for all modules of length strictly smaller than $\ell(M)$ and consider a short exact sequence 
 $$0\to M'\to M\to M''\to 0,$$
 with $M'$ simple. The functor $\Fr$, which is exact and faithful when restricted to $\Ver_{p^n}^{2r}$, by Corollary~\ref{cor:Fr_exact_faithful}, yields a commutative diagram
 $$\xymatrix{
 0\ar[r]&\Hom(\unit,M')\ar[d]\ar[r]&\Hom(\unit,M)\ar[d]\ar[r]&\Hom(\unit,M'')\ar[d]\ar[r]&\Ext^1_{{\Ver_{p^{n}}^{2r}}}(\unit,M')\ar[d]\\
0\ar[r]&\Hom(\unit,\Fr M')\ar[r]&\Hom(\unit,\Fr M)\ar[r]&\Hom(\unit,\Fr M'')\ar[r]&\Ext^1_{{\Ver_{p^{n-1}}^{2r}}}(\unit,\Fr M'),
 }$$
 with exact rows. By the induction hypothesis, the first and third vertical arrow from the left are isomorphisms, and we already know that the second arrow from the left is a monomorphism. It hence suffices to prove that the right morphisms is a monomorphism as well. The equivalent claim about extensions of the form $\Ext^1_{{\Ver_{p^{n}}^{2r}}}(M',\unit)$, for simple $M'\in {\Ver_{p^{n}}^{2r}}$, is proved in Lemma~\ref{Lem:Ext} for $s=2r$.

This completes the proof of the proposition.
\end{proof}

We now prove a general statement about essential surjectivity. In order to do so, we will use the following notion:
\begin{definition}
 A full subcategory $\cB$ of an abelian category $\cC$ will be called {\it generating} if every object in $\cC$ is a subquotient of an object in $\cB$.
\end{definition}
\begin{example}
 The full subcategory of finite direct sums of tensor powers ${\left(L^{[n]}_1\right)^{\otimes j}}$, $0 \leq j \leq r$ is a generating subcategory of the category ${\Ver_{p^{n}}^r}$.
\end{example}

\begin{prop}\label{prop:func_ess_surj}
 Let $F: \cC' \to \cC$ be a full exact functor between finite length abelian categories. Assume the following conditions hold:
 \begin{enumerate}
  \item \label{itms:ess_surj1} $F$ sends simple objects in $\cC'$ to simple or zero objects in $\cC$.
  \item \label{itms:ess_surj2} There exists a generating full subcategory $\cB \subset\cC$ such that $\cB \subset F(\cC')$.
 \end{enumerate}
Then $F$ is essentially surjective.
\end{prop}
\begin{proof}

Denote by $\cA$ the essential image of $\cC$ under $F$. Since $F$ is full and exact, $\cA$ is an abelian subcategory of $\cC$, {\it i.e.} it  is full, and closed under taking kernels and cokernels of morphisms. We claim that the conditions (1) and (2) imply that

(1)' Every simple object in $\cC$ lies in $\cA$. 

Indeed, by assumption (2) every simple object $L\in\cC$ is a subquotient of an object $B\in\cB$ for which there exists $B'\in\cC'$ with $B=F(B')$. Taking a Jordan-H\"older filtration of $B'$ yields, through the exact functor $F$, an exhaustive filtration of $B$. By assumption (1), the quotients of the filtration are simple and hence this is a Jordan-H\"older filtration of $B$. Consequently $L$ is of the form $F(L')$ for some simple $L'\in\cC'$.

   We also note that any object in $\cC$  can be presented as the cokernel of a morphism $M_1 \to M_2$, where $M_1, M_2$ are subobjects of an object in $\cB$. It is therefore enough to prove that $M_1, M_2 \in \cA$.
    
  Now, consider a subobject $M$ of an object $T \in \cB$. We will prove that $M \in \cA$ by induction on the length of $M$.
  
 {\bf Base:} The case $M=0$ is obvious. 
 
  {\bf Step:} Let $0 \neq M \subset T$ with $T\in\cB$. We denote by $i: M \hookrightarrow T$ the corresponding monomorphism.

  Let $M \stackrel{q}{\twoheadrightarrow} L $ be a simple quotient of $M$ and let $N:= Ker(q)$. Then we have a commutative diagram with exact rows:
  $$ \xymatrix{   &0 \ar[r] &N \ar[d]^{\id} \ar[r] &M \ar[r]^-{q} \ar[d]^{i} &{L \cong M/N} \ar[r] \ar[d] &0 \\
    &0 \ar[r] &N \ar[r] &{T} \ar[r] &{T/N} \ar[r] &0 }$$

    We now have: $N \in \cA$ (by induction assumption), $L \in \cA$ by (1)' and $T \in \cA$ by assumption \eqref{itms:ess_surj2}. Since $\cA$ is closed under taking cokernels, $T/N \in \cA$ as well. 
    Now, by standard diagram chasing (e.g. the Snake Lemma), we have an exact sequence 
    $$0 \to L \to T/N \to T/M \to 0$$ where $L \to T/N$ is the map determined by $i$ and $q$.
    Since both $L,\, T/N \in \cA$, we conclude that $T/M \in \cA$, and hence $M \in \cA$ (since it is the kernel of the map $T \to T/M$). 
    
    This completes the proof of the proposition.
\end{proof}

\begin{corollary}\label{cor:Fr_ess_surj} 
  If $0<{2r<p^{n-1} - p^{n-2}}$ the Frobenius twist $\Fr: {\Ver_{p^{n}}^r}\to{\Ver_{p^{n-1}}^r}$ is essentially surjective.
\end{corollary}
\begin{proof}
We only need to check that the assumptions of Proposition \ref{prop:func_ess_surj} are satisfied.

That $\Fr :{\Ver_{p^{n}}^r}\to{\Ver_{p^{n-1}}^r}$ sends simple objects to simple objects follows from equation~\eqref{FrSimple}.
  
Second, consider the full subcategory $\cB \subset {\Ver_{p^{n-1}}^r}$ of finite direct sums of tensor powers ${\left(L_1^{[n-1]}\right)^{\otimes j}}$, $0\leq j \leq r$. For each $j \leq r$, $$\left(L_1^{[n-1]}\right)^{\otimes j} = \Fr\left(\left(L_1^{[n]}\right)^{\otimes j}\right)$$ hence $\cB \subset \Fr\left({\Ver_{p^{n}}^r}\right)$, as required.
\end{proof}

\subsection{The category \texorpdfstring{$\mathcal{C}$}{Ver_p(infty)}}

Given $r>0$, Theorem \ref{thm:Fr_loc_equiv} allows us to define a (stabilizing) limit of the system of categories and functors $(\Ver_{p^{n}}^r, \Fr)$ as $n \to \infty$. This limit will be denoted by $$\mathcal{C}^r = \varprojlim_{n\to \infty} {\Ver_{p^{n}}^r}.$$
In particular, we have $\mathcal{C}^r\simeq \Ver_{p^{n}}^r$ if $n>4r$, so $\cC^r$ is a $k$-linear abelian category, and we have obvious (fully faithful, exact) embeddings $\mathcal{C}^r \hookrightarrow \cC^{r+1}$ and hence we can define a colimit $$\mathcal{C}:=\varinjlim_{r \to\infty} \mathcal{C}^r.$$ 


The abelian category $\mathcal{C}$ inherits from the categories $\Ver_{p^{n}}$ a rigid symmetric monoidal structure, given by bifunctors $$- \otimes -: \mathcal{C}^r \times \mathcal{C}^{r'}\longrightarrow \mathcal{C}^{r+r'}. $$ 
Hence $\mathcal{C}$ is a tensor category in the sense of Section~\ref{DefTC}. It possesses a distinguished object $\bar{L}_1$ which is the limit of the objects $L_1^{[n]} \in \Ver_{p^{n}}^r$. By construction, $\mathcal{C}$ is tensor-generated by $\bar{L}_1$: namely, any object $M \in \mathcal{C}$ is a subquotient of a finite direct sum of tensor powers of $\bar{L}_1$.

Recall that the simple objects in $\Ver_{p^n}$ are labelled by $0\le a \le p^n-p^{n-1}$ as in \eqref{LabelSimp}. Since $\cC^r\hookrightarrow\cC^{r+1}$ is the embedding of a topologising subcategory, it follows that the set of isomorphism classes of simple objects can be interpreted as the union of the corresponding sets for $\cC^r$. By Corollary~\ref{Cor:Simp}(1), we can therefore label the simple objects in $\cC$ by
$$\{\bar{L}_a\,|\, a\in\mN\},$$
where $\bar{L}_a$ for $a\le r$, corresponds to the simple object $L_{a}^{[n]}$ in $\cC^r\subset\cC$, well-defined by \eqref{FrSimple}.

\begin{lemma}\label{LemFrL}
The Frobenius functor $\Fr$ on $\cC$ takes values in $\cC\subset\cC\boxtimes\Ver_p$, and
$$\Fr(\bar{L}_a)\cong \bar{L}_{pa},$$
for every $a\in \mN$.
\end{lemma}
\begin{proof}
By Lemma~\ref{StupidLemma}, we can calculate the Frobenius twist using 
$\cC^r\subset\cC$, for $r>>{0}$, with 
$F:\cC^r=\Ver_{p^n}^r\hookrightarrow \Ver_{p^n}$. The conclusion then follows 
from equation~\eqref{FrSimple}.
\end{proof}

\begin{lemma}\label{lem:tilt_ff}
 There exists a fully faithful $k$-linear SM functor $I':Tilt(SL_2) \to \mathcal{C}$ sending the standard $2$-dimensional $SL_2$-representation $V$ to $\bar{L}_1$.
\end{lemma}
\begin{proof}

Consider the object $\bar{L}_1 \in \mathcal{C}$. By Theorem 
\ref{thm:Fr_loc_equiv} we have for $n >> r$ an exact faithful functor
\[ F_n: \Ver_{p^{n}}^r\cong  \mathcal{C}^r \hookrightarrow \mathcal{C}\] 
which maps $L_{1}^{[n]}$ to $\bar{L}_1$. Since the monoidal structure on $\mathcal{C}$ is given by the bifunctors $- \otimes -: \mathcal{C}^r \times \mathcal{C}^{r'}\to\mathcal{C}^{r+r'}$, $(L_1^{[n]})^{\otimes 2}$ is sent to $\bar{L}_1^{\otimes 2}$. {By definition of $\wedge^2(\bar{L}_1)$ we have in $\mathcal{C}$ 
\[ \wedge^2(\bar{L}_1) = Im(\gamma_{\bar{L}_1\bar{L}_1} - 1: \bar{L}_1^{\otimes 2} \to \bar{L}_1^{\otimes 2}).\] 

Since $\bar{L}_1 \cong F_n(L_1^{[n]})$ and $\bar{L}_1^{\otimes 2} \cong F_n((L_1^{[n]})^{\otimes 2})$ we get
\[ \gamma_{\bar{L}_1\bar{L}_1} = \gamma_{F_n(L_1^{[n]}) F_n(L_1^{[n]})} = F_n (\gamma_{L_1^{[n]} L_1^{[n]}}),\]
the latter by construction of the symmetric monoidal structure on $\mathcal{C}$. This shows that
\[ F_n(\wedge^2(L_1^{[n]})) \cong \wedge^2(F_n(L_1^{[n]})) \cong \wedge^2(\bar{L}_1)\]
and therefore that $\wedge^2(L_1^{[n]})\cong \unit$ implies $\wedge^2(\bar{L}_1) \cong \unit$. In a similar way one can show $\wedge^3(\bar{L}_1) \cong 0$ and $\dim(\bar{L}_1) = 2$.} 
 
Therefore, by \cite[Proposition 3.4]{BE}, we have an additive symmetric monoidal functor $I':Tilt(SL_2) \to \mathcal{C}$ with $I'(V) \cong \bar{L}_1$. To see that it is fully faithful, recall that it is enough to check this on tensor powers of $V$ (these generate $Tilt(SL_2)$ under taking finite direct sums and direct summands). 
 
 For any $r\geq 1$, denote by $Tilt^r(SL_2)$ the full subcategory generated by objects $V^{\otimes k}$, $k\leq r$, under taking finite direct sums and direct summands.
 
 Then the essential image of $Tilt^r(SL_2)$ under $I'$ clearly lies in $\mathcal{C}^r$, and the functor $I': Tilt^r(SL_2) \longrightarrow \mathcal{C}^r$ is the restriction of the additive symmetric monoidal functor $$F_n: Tilt(SL_2) \to \Ver_{p^{n}},\;\; V \mapsto L_1^{[n]}, \;\; n>>r,$$ 
 from the definition of $\Ver_{p^{n}}$ as the abelian envelope of $Tilt(SL_2)/\mathcal{I}_{n}$, see \cite{BEO}.
  In particular $F_n$ is full and its kernel is the ideal $\mathcal{I}_{n}$ from Section~\ref{SecC2n}. Since we are considering the limit $n \to \infty$, we conclude that the functor $I': Tilt^r(SL_2) \longrightarrow \mathcal{C}^r$ is fully faithful for every $r\geq 1$. The statement of the lemma now follows.
\end{proof}

The  case $p=2$ of the following theorem proves a conjecture of Benson and Etingof. 
\begin{theorem}\label{thm:BE_conj}
 The functor $I':Tilt(SL_2) \to \mathcal{C}$ factors through the functor $I:Tilt(SL_2) \to \Rep(SL_2)$ and induces an equivalence of tensor categories $$\Rep(SL_2) \to \mathcal{C}.$$
\end{theorem}
\begin{proof}
 By Theorems \ref{thm:tilt_ab_env} and \ref{thm:uni-2}, the functor $I'$ factors through the functor $I$. The induced functor $\Phi:\Rep(SL_2) \to \mathcal{C}$ is a faithful exact symmetric monoidal functor sending the standard $SL_2$-representation $V$ to $\bar{L}_1$. Since $I'$ is full, Theorem \ref{thm:uni-2} implies that $\Phi$ is full as well.
  
 
 

It now suffices to prove that $\Phi$ is essentially surjective. By Proposition~\ref{prop:func_ess_surj} we only need to show that $\Phi$ sends simple objects to simple objects. 





By \cite[Lemma 3.10]{BEO} $T_i$ is simple for $i \leq p-1$ and $L_1 \otimes L_i \cong L_{i+1} \oplus L_{i-1}$ for $i \leq p-2$. Similarly $\mT^{[n]}_{i} \otimes \mT^{[n]}_{1} \cong \mT^{[n]}_{i+1} \oplus \mT^{[n]}_{i-1}$ in $\Ver_{p^n}$ by \cite[Proposition 4.49]{BEO} for $i \leq p-2$. In particular this holds for $n >> r$ and therefore in $\mathcal{C}^r$ and then in $\mathcal{C}$ for the corresponding limit objects. We conclude $\bar{L}_i\otimes \bar{L}_1 \cong \bar{L}_{i+1} \oplus \bar{L}_{i-1}$ in $\mathcal{C}$ for $i \leq p-2$. Since $\Phi(L_1) \cong \bar{L}_1$ we can therefore iteratively conclude $\Phi(L_i) = \bar{L}_i$ for $0 \leq i \leq p-1$.

Since $\Phi$ is a tensor functor, it commutes with $\Fr$, see \cite[\S 5.1]{EO}. As explained in Section~\ref{ssec:Fr}, we have $\Fr^a(L_i)\cong L^{(a)}_i\cong L_{p^ai}$, for all positive integers $a$. On the other hand, by Lemma~\ref{LemFrL}, we have $\Fr^a(\bar{L}_i)\cong \bar{L}_{p^ai}$. By the above paragraph, we thus find $\Phi(L_{p^ai}) = \bar{L}_{p^ai}$ for all $a>0$ and $0\le i\le p-1$.

Steinberg's tensor product theorem \cite[Corollary 3.17]{Jantzen} states that $L_a\cong \otimes_{i}L_{a_i}^{(i)}$ for a $p$-adic expansion $a = \sum_i a_i p^i$. It then follows from \eqref{LabelSimp} that $\Phi(L_a)\cong \bar{L}_a$ for all $a\in\mN$ and the theorem is proved.
\end{proof}



\section*{Acknowledgements}
The research of K. C. was supported by ARC grant DE170100623. The research of I. E. and T. H. was supported by the ISF grant 711/18. The research of T. H. was partially funded by the Deutsche Forschungsgemeinschaft (DFG, German Research
Foundation) under Germany's Excellence Strategy – EXC-2047/1 – 390685813. We thank the organisers of the conference ``Representations Theory in Venice'' (September 2019), where much of the work on this paper was done, for their hospitality. I. E. would like to thank N. Harman and V. Ostrik for helpful discussions. T. H. would like to thank W. Hardesty and P. Sobaje for background on tilting modules.

\end{document}